\documentclass[12pt,reqno]{article}

\usepackage{amsmath,amsthm,amsfonts,amssymb,latexsym,mathrsfs,color}

\newtheorem{thm}{Theorem}[section]

\newtheorem{prop}[thm]{Proposition}
\newtheorem{conj}[thm]{Conjecture}

\newtheorem{coro}[thm]{Corollary}
\theoremstyle{definition}
\newtheorem{de}[thm]{Definition}
\newtheorem{exm}[thm]{Example}
\theoremstyle{remark}
\newtheorem{rem}[thm]{Remark}

\def \R{\mathbb{R}}

\def \A{\mathfrak{A}}
\def \B{\mathfrak{B}}
\def \S{\mathfrak{S}}
\def \F{\mathfrak{F}}
\def \G{\mathfrak{G}}

\newcommand{\cb}{{\mathscr B}}
\newcommand{\cc}{{\mathscr C}}

\newcommand{\La}{\Lambda}

\newcommand{\cp}{\mathcal{P}}

\def\on{\overline{n}}

\newcommand{\lrf}[1]{\left\lfloor #1\right\rfloor}

\def \dar{{\rm dar}\,}

\title{Polynomials with palindromic and unimodal coefficients
\thanks{Partially supported by the National Natural Science Foundation of China (Grant Nos. 11071030, 11371078)
and the Specialized Research Fund for the Doctoral Program of Higher Education of China (Grant No. 20110041110039).}}
\author{Hua Sun,\quad Yi Wang
\thanks{Corresponding Author.
\newline\hspace*{5mm}
{\it Email address:}\quad wangyi@dlut.edu.cn (Y. Wang)}
,\quad Hai-Xia Zhang}
\date{\footnotesize School of Mathematical Sciences,
         Dalian University of Technology,
         Dalian 116024, PR China}

\begin{document}
\maketitle
\begin{abstract}
Let $f(q)=a_rq^r+\cdots+a_sq^s$, with $a_r\neq 0$ and $a_s\neq 0$,
be a real polynomial. It is a palindromic polynomial of darga $n$ if
$r+s=n$ and $a_{r+i}=a_{s-i}$ for all $i$. Polynomials of darga $n$
forms a linear subspace $\mathcal{P}_n(q)$ of $\mathbb{R}(q)_{n+1}$
of dimension $\lfloor{n/2}\rfloor+1$. We give transition matrices
between two bases $\left\{q^j(1+q+\cdots+q^{n-2j})\right\},
\left\{q^j(1+q)^{n-2j}\right\}$ and the standard basis
$\left\{q^j(1+q^{n-2j})\right\}$ of $\mathcal{P}_n(q)$. We present
some characterizations and sufficient conditions for palindromic
polynomials that can be expressed in terms of these two bases with
nonnegative coefficients.
We also point out the link between such polynomials and rank-generating functions of posets.
\\
{\sl MSC:}\quad 05A20; 05A15; 15A03; 06A07
\\
{\sl Keywords:}\quad Unimodal sequence; Palindromic sequence; Linear space; Poset
\end{abstract}

\section{Introduction}
\hspace*{\parindent}

Let $a_0,a_1,\ldots,a_n$ be a sequence of real numbers.
It is {\it palindromic} (or {\it symmetric}) with center of symmetry at $n/2$ if $a_i=a_{n-i}$ for all $i$.
It is {\it unimodal} if $a_0\le\cdots\le a_{m-1}\le a_m\ge a_{m+1}\ge\cdots\ge a_n$ for some $m$.

Let $f(q)=a_rq^r+\cdots+a_sq^s$, with $a_r\neq 0$ and $a_s\neq 0$, be a real polynomial.
Following Zeilberger~\cite{Zei89a},
the {\it darga} of $f(q)$, denoted by $\dar f$, is defined to be $r+s$, i.e., the sum of its lowest and highest powers.
For example, $\dar (1+q)=1$ and $\dar (q)=2$.
We say that the polynomial $f(q)$ is {\it palindromic} ({\it unimodal}, resp.)
if the sequence $a_r,a_{r+1},\ldots,a_{s-1},a_s$ has the corresponding property.
Following Brenti~\cite{Bre90},
we say that a polynomial with nonnegative coefficients is a {\it $\La$-polynomial}
if it is both palindromic and unimodal.

$\La$-polynomials arise often in algebra, analysis, and
combinatorics. Two most well-known $\La$-polynomials are
$1+q+\cdots+q^n$ and $(1+q)^n=\sum_{i=0}^n\binom{n}{i}q^i$. Another
famous $\La$-polynomial is the Gaussian polynomial defined by
$$\binom{n}{i}_q=\frac{(q^n-1)(q^{n-1}-1)\cdots(q^{n-i+1}-1)}{(q^i-1)(q^{i-1}-1)\cdots(q-1)}.$$
The symmetry of the Gaussian polynomial follows immediately by
definition. But the unimodality of the Gaussian polynomial is not so
easy to prove. The first proof of the unimodality was given by
Sylvester~\cite{Syl} in 1878 using the classical theory of
invariants. Then there have been several different proofs based on
methods of Lie algebras, linear algebra, algebraic geometry or
P\'{o}lya theory~\cite{Pro82}. In 1990, K. O'Hara~\cite{Koh90} gave
her celebrated constructive proof, which, roughly speaking, is to
decompose the Gaussian polynomial into a finite sum of suitable
palindromic polynomials of the same darga~\cite{Zei89b}.

To prove the unimodality of a polynomial is often a very difficult task.
The case for palindromic polynomials is somewhat different.
$\La$-polynomials have much better behavior than unimodal polynomials.
For example, the product of two unimodal polynomials is not unimodal in general;
however, the product of two $\La$-polynomials is still a $\La$-polynomial \cite[Theorem 1]{And75}.
This is the motivation for us to study palindromic polynomials.
The object of the present paper is to explore such a topic from a viewpoint of linear algebra.

The organization of the paper is as follows.
We first study algebraic properties of palindromic polynomials in the next section.
Palindromic polynomials of darga $n$ form a linear space of dimension $\lrf{n/2}+1$.
We give transition matrices between the bases
$\left\{q^j(1+q+\cdots+q^{n-2j})\right\}, \left\{q^j(1+q)^{n-2j}\right\}$ and the standard basis $\left\{q^j(1+q^{n-2j})\right\}$.
Then in Section 3 we study palindromic polynomials that can be expressed
in terms of the first two bases with nonnegative coefficients respectively.
Certain known results can be obtained and extended in a unified approach.
Finally in Section 4
We point out the link between palindromic polynomials and rank-generating functions of posets.

Throughout this paper all polynomials considered are real.
As usual, denote by $\mathbb{R}[q]_{n+1}$ the set of real polynomials of degree at most $n$,
and let the binomial coefficient $\binom{n}{i}=0$ unless $0\le i\le n$.
For the sake of convenient, denote $\on=\lrf{n/2}$ and $0$ is viewed as a palindromic polynomial (of any darga).

\section{Palindromic polynomials}
\hspace*{\parindent}

We study algebraic properties of palindromic polynomials in this section.
Some of them are obvious and have occurred in the literature.

Let $f(q)$ be a polynomial of darga $n$.
The reciprocal polynomial $f^*$ of $f$ is defined by
$$f^*(q)=q^nf\left(\frac{1}{q}\right).$$
In other words, if $f(q)=a_rq^r+\cdots+a_sq^s$, where $a_r\neq 0, a_s\neq 0$ and $r+s=n$, then $f^*(q)=a_sq^r+\cdots+a_rq^s$.
%It is clear that $(f^*)^*(q)=f(q)$ and $[f(q)g(q)]^*=f^*(q)g^*(q)$.
The following characterization for palindromic polynomials is direct.

\begin{prop}\label{sym-c}
A polynomial $f(q)$ is palindromic if and only if $f^*(q)=f(q)$.
\end{prop}

\begin{coro}\label{lv}
Suppose that $a\in\R$, $f(q)$ and $g(q)$ are two palindromic polynomials of darga $n$.
Then $af(q)$ and $f(q)+g(q)$ are palindromic polynomials of darga $n$.
\end{coro}

It is clear that $[f(q)g(q)]^*=f^*(q)g^*(q)$. So the following result is immediate.

\begin{coro}\label{fg}
Let $f(q)$ and $g(q)$ be two palindromic polynomials of darga $n$ and $m$ respectively. Then
\begin{enumerate}
  \item [\rm (i)] $f(q)g(q)$ is a palindromic polynomial of darga $n+m$;
      in particular, $q^jf(q)$ is a palindromic polynomial of darga $n+2j$.
  \item [\rm (ii)] If $g(q)|f(q)$ and $f(q)=g(q)h(q)$, then $h(q)$ is also palindromic. %polynomial of darga $n-m$.
\end{enumerate}
\end{coro}

Let $f(q)$ be a palindromic polynomial.
If $z_0\neq 0$ is a zero of $f(q)$, then $1/z_0$ is also a zero of $f(q)$ and has the same multiplicity.
On the other hand, the complex zeros of a real polynomial occur in conjugate pairs.
So a irreducible palindromic polynomial must be linear, quadratic or quartic.

\begin{prop}\label{factor}
Let $f(q)\in\mathbb{R}^+[q]$.
Then $f(q)$ is palindromic if and only if %it can be factored into %an expression with the form
$$f(q)=aq^r(1+q)^e\prod_{i=1}^k\left[(1+r_iq)(r_i+q)\right]^{e_i}\prod_{j=1}^{\ell}\left[(1+b_jq+c_jq^2)(c_j+b_jq+q^2)\right]^{\delta_j},$$
where $r_i\neq 1$ and $b_j^2<4c_j$.
\end{prop}

\begin{coro}
Let $f(q)\in\mathbb{R}^+[q]$.

{\rm (i)}\quad
If $f(q)$ is a palindromic polynomial of odd dagar,
then $(1+q)|f(q)$.

{\rm (ii)}\quad
If the modulus of each zero of $f(q)$ is $1$,
%If $f(q)$ has only unimodular zeros,
then $f(q)$ is palindromic.
\end{coro}

\begin{de}
We say that $\{f_j(q)\}_{j\ge 0}$ is {\it a sequence of basic palindromic polynomials}
if $f_j(q)$ is a palindromic polynomial of dagar $j$ with nonzero constant term.
\end{de}

\begin{exm}\label{bsp}\rm
Three sequences of basic palindromic polynomials:
\begin{itemize}
  \item [\rm (i)] $S_0(q)=1$ and $S_j(q)=1+q^j$ for $j\ge 1$.
  \item [\rm (ii)] $A_0(q)=1$ and $A_j(q)=1+q+\cdots+q^j$ for $j\ge 1$.
  \item [\rm (iii)] $B_0(q)=1$ and $B_j(q)=(1+q)^j$ for $j\ge 1$.
\end{itemize}
\end{exm}

Let $\cp_n(q)$ denote the set of palindromic polynomials of darga $n$.
Then $\cp_n(q)$ forms a linear subspace by Corollary~\ref{lv}.
A sequence of basic palindromic polynomials can induce a basis of $\cp_n(q)$.
More precisely, we have the following.

\begin{thm}\label{sym}
Let $\{f_j(q)\}_{j\ge 0}$ be a sequence of basic palindromic polynomials.
Assume that $f(q)$ is a real polynomial of darga $n$.
Then $f(q)$ is palindromic if and only if $f(q)$ can be written as
\begin{equation}\label{exp}
    \sum_{j=0}^{\on}c_jq^jf_{n-2j}(q),\qquad c_j\in\mathbb{R}.
\end{equation}
In other words, $\cp_n(q)$ is a linear space of dimension $\on+1$,
with a basis
$$\F_n=\left\{q^jf_{n-2j}(q): j=0,1,\ldots, \on\right\}.$$
\end{thm}
\begin{proof}
The ``if" part follows from Corollary~\ref{lv}
since all terms $q^jf_{n-2j}(q)$ in the summation (\ref{exp}) are palindromic polynomials with the same darga $n$.
We next prove the ``only if" part by induction on $\dar f$.
We may assume, without loss of generality, that all $f_j$ are monic.

Let $f(q)=a_rq^r+\cdots+a_sq^s$ be a palindromic polynomial of darga $n$,
where $a_r=a_s\neq 0$ and $r+s=n$.
Define
$$g(q)=\frac{f(q)-a_rq^rf_{s-r}(q)}{q}.$$
Then $g(q)$ is a palindromic polynomial of darga $\le n-2$.
By the induction hypothesis,
$$g(q)=\sum_{j=0}^{\on-1}b_jq^jf_{n-2-2j}(q),\qquad b_j\in\R.$$
Hence
$$f(q)=a_rq^rf_{s-r}(q)+\sum_{j=0}^{\on-1}b_jq^{j+1}f_{n-2-2j}(q)=\sum_{j=0}^{\on}c_jq^jf_{n-2j}(q),$$
where $c_j\in\R$ for all $j$.
This completes the proof.
\end{proof}

We next consider the transition matrix between two bases of the linear space $\mathcal{P}_n(q)$.

\begin{thm}\label{trans-matrix}
Let $\{f_j(q)\}_{j\ge 0}$ and $\{g_j(q)\}_{j\ge 0}$ be two sequences of basic palindromic polynomials.
Suppose that
$$g_j(q)=\sum_{k=0}^{\lrf{j/2}}t(j,k)q^kf_{j-2k}(q).$$
Then the transition matrix from the basis $\F_n=\{q^jf_{n-2j}(q)\}_{j=0}^{\on}$ of $\mathcal{P}_n(q)$ to the basis $\G_n=\{q^jg_{n-2j}(q)\}_{j=0}^{\on}$
is the $(\on+1)\times (\on+1)$ lower triangle matrix
$$M(\F_n,\G_n)=\left[
   \begin{array}{ccccc}
     t(n,0) & 0 & 0 & \cdots &  0\\
     t(n,1) & t(n-2,0) & 0& \cdots & 0 \\
     t(n,2) & t(n-2,1) & t(n-4,0) & \cdots  & 0 \\
     \vdots & \vdots & \vdots & \vdots  & \vdots  \\
     t(n,\on) & t(n-2,\on-1) &  t(n-4,\on-2) &\cdots  & t(n-2\on,0) \\
   \end{array}
 \right],$$
where the $(i,j)$ entry of $M(\F_n,\G_n)$ is $t(n-2j,i-j)$ for $0\le
j\le i\le \on$.
\end{thm}
\begin{proof}
Note that
$$q^jg_{n-2j}(q)=\sum_{k=0}^{\lrf{(n-2j)/2}}t(n-2j,k)q^{j+k}f_{n-2j-2k}(q)=\sum_{i=j}^{\on}t(n-2j,i-j)q^{i}f_{n-2i}(q)$$
for $j=0,1,\ldots,\on$.
Hence the statement follows immediately.
\end{proof}

Let us examine transition matrices between the following three bases %of the linear space $\mathcal{P}_n(q)$
induced by sequences of basic palindromic polynomials in Example~\ref{bsp}:
\begin{itemize}
  \item [\rm (i)] $\S=\{q^jS_{n-2j}(q)\}_{j=0}^{\on}$, where $S_0(q)=1$ and $S_j(q)=1+q^j$ for $j\ge 1$.
  \item [\rm (ii)] $\A=\{q^jA_{n-2j}(q)\}_{j=0}^{\on}$, where $A_0(q)=1$ and $A_j(q)=1+q+\cdots+q^j$ for $j\ge 1$.
  \item [\rm (iii)] $\B=\{q^jB_{n-2j}(q)\}_{j=0}^{\on}$, where $B_0(q)=1$ and $B_j(q)=(1+q)^j$ for $j\ge 1$.
\end{itemize}

\begin{coro}
The transition matrices
$$M(\S,\A)=\left[
   \begin{array}{ccccc}
     1 &  &  &  &  \\
     1 & 1 &  & O &  \\
     \vdots & \vdots & \ddots &  &  \\
     1 & 1 & \cdots & 1 &  \\
     1 & 1 & \cdots & 1 & 1 \\
   \end{array}
 \right]
$$
and
$$M(\S,\B)=(b_{i,j})_{(\on+1)\times (\on+1)}
=\left[
   \begin{array}{cccccc}
     1 &  &  &  & & \\
     \binom{n}{1} & 1 & && O   &  \\
     \binom{n}{2} & \binom{n-2}{1} & 1 & &  &  \\
     \vdots & \vdots & \vdots & \ddots  & &  \\
     \binom{n}{\on-1} & \binom{n-2}{\on-2} & \binom{n-4}{\on-3} & \cdots & 1 &  \\
     \binom{n}{\on} & \binom{n-2}{\on-1} &  \binom{n-4}{\on-2} &\cdots & \binom{n-2\on+2}{1} & 1 \\
   \end{array}
 \right],$$
respectively, where $b_{i,j}=0$ for $i<j$ and $b_{i,j}=\binom{n-2j}{i-j}$ for $0\le j\le i\le \on$.
\end{coro}
\begin{proof}
Note that
$$1+q+\cdots+q^n=(1+q^n)+q(1+q^{n-2})+q^2(1+q^{n-4})+\cdots$$
and
$$(1+q)^n=(1+q^n)+\binom{n}{1}q(1+q^{n-2})+\binom{n}{2}q^2(1+q^{n-4})+\cdots.$$
Hence the statement follows from Theorem~\ref{trans-matrix}.
\end{proof}

Other transition matrices between these three bases can be deduced from $M(\S,\A)$ and $M(\S,\B)$.
For example,
$$%\label{A2S}
M(\A,\S)=M^{-1}(\S,\A)=\left[
   \begin{array}{rrrrr}
     1 &  &  &  &  \\
     -1 & 1 &  & O &  \\
     0 & -1 & 1 & & \\
     \vdots & \vdots  & \ddots & \ddots &  \\
     0 & 0 & \cdots & -1 & 1 \\
   \end{array}
 \right],
$$
and
$$M(\A,\B)=M(\A,\S)M(\S,\B)=(d_{i,j})_{(\on+1)\times (\on+1)},$$
where $d_{i,j}=0$ for $i<j$ and $d_{i,j}=b_{i,j}-b_{i-1,j}=\binom{n-2j}{i-j}-\binom{n-2j}{i-j-1}$
for $0\le j\le i\le \on$.
It is not difficult to verify that all $d_{i,j}$ are positive for $0\le j\le i\le \on$.

We can also obtain $M(\B,\S)$ by another approach.

\begin{coro}\label{B2S}
Let $M(\B,\S)=(c_{i,j})_{(\on+1)\times (\on+1)}$.
Then $c_{i,j}=0$ for $i<j$ and $$c_{i,j}=(-1)^{i-j}\frac{n-2j}{n-i-j}\binom{n-i-j}{i-j}$$
for $0\le j\le i\le \on$.
\end{coro}
\begin{proof}
Recall the classical Chebyshev inversion relation
\begin{equation}\label{Cir}
x_n=\sum_{i=0}^{\on} \binom{n}{i}y_{n-2i}\Longleftrightarrow y_n=\sum_{i=0}^{\on} C(n,i) x_{n-2i},
\end{equation}
where $$C(n,i)=(-1)^i \frac{n}{n-i}\binom{n-i}{i}$$ for $0\le i\le \on$
(see Riordan~\cite[p.54]{Rio68} for instance).
Now
$$B_n(q)=\sum_{i=0}^{\on}\binom{n}{i}q^iS_{n-2i}(q),$$
which can be rewritten as
$$q^{-n/2}B_n(q)=\sum_{i=0}^{\on}\binom{n}{i}q^{-(n-2i)/2}S_{n-2i}(q).$$
So by (\ref{Cir}),
$$q^{-n/2}S_n(q)=\sum_{i=0}^{\on}\binom{n}{i}q^{-(n-2i)/2}B_{n-2i}(q),$$
i.e.,
$$S_n(q)=\sum_{i=0}^{\on}C(n,i) q^iB_{n-2i}(q).$$
Thus by Theorem~\ref{trans-matrix},  $c_{i,j}=0$ for $i<j$ and
$$c_{i,j}=C(n-2j,i-j)=(-1)^{i-j}\frac{n-2j}{n-i-j}\binom{n-i-j}{i-j}$$ for $0\le j\le i\le \on$.
This completes our proof.
\end{proof}

\begin{rem}\rm
Note that $$C(n,i)=(-1)^{i}\left[\binom{n-i}{i}+\binom{n-i-1}{i-1}\right].$$
Hence all $C(n,i)$ are integers, and so are all $c_{i,j}$.
Thus $M(\B,\S)$ is an integer matrix.
\end{rem}

\section{$\La$-polynomials}
\hspace*{\parindent}

Let $\F$ be a basis of $\cp_n(q)$ and $f(q)\in\cp_n(q)$.
We say that $f(q)$ is {\it $\F$-positive} if $f(q)$ can be expressed in terms of the basis $\F$ with nonnegative coefficients.
Clearly, the set of all $\F$-positive polynomials forms a convex cone of $\cp_n(q)$.
Our concern in this section is $\A$-positive polynomials and $\B$-positive polynomials.
A $\B$-positive polynomial is usually said to have nonnegative $\gamma$-vector
and a $\A$-positive polynomial is precisely a $\La$-polynomial.
%Throughout this section polynomials considered are assumed to have positive coefficients.
%The following characterization for $\La$-polynomials is folklore.
%We give its proof for the sake of completeness.

\begin{thm}\label{iff}
Let $f(q)$ be a palindromic polynomial of nonnegative coefficients.
Then $f(q)$ is $\A$-positive if and only if $f(q)$ is unimodal,
i.e., $f(q)$ is a $\La$-polynomial.
\end{thm}
\begin{proof}
Let $f(q)$ be a palindromic polynomial of darga $n$. Assume that
$$f(q)=\sum_{j=0}^{\on}a_jq^j(1+q^{n-2j})=\sum_{j=0}^{\on}c_jq^j(1+q+\cdots+q^{n-2j}).$$
Then $c_0=a_0$ and $c_j=a_{j}-a_{j-1}$ for $1\le j\le \on$ by the transition matrix $M(\A,\S)$.
Thus $c_j$ is nonnegative if and only if $a_j$ is nondecreasing for $0\le j\le \on$.
In other words, $f(q)$ is $\A$-positive if and only if $f(q)$ is unimodal.
\end{proof}

The following result is an immediate consequence of Corollary~\ref{B2S}

\begin{thm}\label{thm-t31}
Let $f(q)=\sum_{j=0}^{n} a_jq^j$ be a palindromic polynomial of darga $n$ and
$$f(q)=\sum_{i=0}^{\on}b_iq^i(1+q)^{n-2i}.$$
\begin{itemize}
  \item [\rm (i)] For $0\le i\le \on$, we have $$b_i=\sum_{j=0}^{i}(-1)^{i-j}\frac{n-2j}{n-i-j}\binom{n-i-j}{i-j}a_j.$$
  \item [\rm (ii)] If all $a_j$ are integers, then so are all $b_i$.
  \item [\rm (iii)] A palindromic polynomial $f(q)=\sum_{j=0}^{n} a_jq^j$ is $\B$-positive if and only if
  $$\sum_{j=0}^{i}(-1)^{i-j}\frac{n-2j}{n-i-j}\binom{n-i-j}{i-j}a_j>0$$ for all $0\le i\le \on$.
\end{itemize}
\end{thm}

Clearly, if all elements of a certain transition matrix from the basis $\F$ to the basis $\F'$ are nonnegative,
then a $\F'$-positive polynomial must be $\F$-positive.
In particular, if $f(q)$ is $\B$-positive, then it is a $\La$-polynomial.

Let $\{f_j(q)\}_{j\ge 0}$ be a sequence of basic palindromic polynomials.
Then
$$\F_n=\left\{ q^jf_{n-2j}(q): j=0,1,\ldots, \on\right\}$$
is a basis of $\cp_n(q)$ by Theorem~\ref{sym}.
For any $i,j\ge 0$, the product $f_i(q)f_j(q)$ is a palindromic polynomial of darga $i+j$,
and so it can be expressed in terms of the basis $\F_{i+j}$.
We say that the sequence $\{f_j(q)\}_{j\ge 0}$ is {\it self-positive} if
$f_i(q)f_j(q)$ is $\F_{i+j}$-positive for any $i,j\ge 0$.

\begin{exm}\rm
Let us examine the self-positivity of three sequences of basic palindromic polynomials in Example~\ref{bsp}.
The sequences $\{S_j(q)\}$ and $\{B_j(q)\}$ are obviously self-positive.
Note that
$$A_i(q)A_j(q)=(1+q+\cdots+q^i)(1+q+\cdots+q^j)=\sum_{k=0}^{\lrf{(i+j)/2}}q^k(1+q+\cdots+q^{i+j-2k}).$$
Hence the sequence $\{A_j(q)\}$ is also self-positive.
\end{exm}

\begin{prop}\label{prop-p}
Let $\{f_j(q)\}_{j\ge 0}$ be a self-positive sequence of basic palindromic polynomials.
Suppose that $f(q)$ and $g(q)$ are $\F_n$-positive and $\F_m$-positive respectively.
Then $f(q)g(q)$ is $\F_{n+m}$-positive.
\end{prop}
\begin{proof}
By the self-positivity of $\{f_j(q)\}_{j\ge 0}$,
we have %$f_{n-2i}(q)f_{m-2j}(q)$ is $\cf_{n+m-2i-2j}$-positive:
$$f_{n-2i}(q)f_{m-2j}(q)=\sum_{k=0}^{\lrf{(n+m-2i-2j)/2}}a_kq^kf_{n+m-2i-2j-2k}(q),\qquad a_k\ge 0.$$
Now suppose that
$$f(q)=\sum_{i=0}^{\lrf{n/2}}b_iq^if_{n-2i}(q),\qquad b_i\ge 0$$ and
$$g(q)=\sum_{j=0}^{\lrf{m/2}}c_jq^jf_{m-2j}(q),\qquad c_j\ge 0.$$ Then
\begin{eqnarray*}
  f(q)g(q) %&=& \sum_{i=0}^{\lrf{n/2}}\sum_{j=0}^{\lrf{m/2}}b_ic_jq^{i+j}f_{n-2i}(q)f_{m-2j}(q) \\
   &=& \sum_{i=0}^{\lrf{n/2}}\sum_{j=0}^{\lrf{m/2}}\sum_{k=0}^{\lrf{(n+m-2i-2j)/2}}a_kb_ic_jq^{i+j+k}f_{n+m-2i-2j-2k}(q)\\
   &=& \sum_{i=0}^{\lrf{n/2}}\sum_{j=0}^{\lrf{m/2}}\sum_{r=i+j}^{\lrf{(n+m)/2}}a_{r-i-j}b_ic_jq^{r}f_{n+m-2r}(q).
\end{eqnarray*}
Thus $f(q)g(q)$ is $\F_{n+m}$-positive,
and the proof is therefore complete.
\end{proof}

In particular, taking $f_j(q)=A_j(q)$ in Proposition~\ref{prop-p}, we obtain a result due to Andrews~\cite[Theorem 1]{And75}.

\begin{coro}\label{p-fg}
The product of $\La$-polynomials is a $\La$-polynomial.
\end{coro}

\begin{rem}\label{prod-3}\rm
%The product of polynomials expressed in terms of the basis $\B_3$ is similar to that of polynomials expressed in terms of the basis $\B_1$.
Let $f(q)=\sum_{i=0}^{\lrf{n/2}}a_iq^i(1+q)^{n-2i}$ and $g(q)=\sum_{j=0}^{\lrf{m/2}}b_jq^j(1+q)^{m-2j}$.
Then $$f(q)g(q)=\sum_{k=0}^{\lrf{(n+m)/2}}c_kq^k(1+q)^{n+m-2k},$$
where $c_k=\sum_{i+j=k}a_ib_j$.
In other words, $\{c_k\}$ is the convolution of $\{a_i\}$ and $\{b_j\}$.
%In particular, the product of two B-polynomials is still a B-polynomial.
This result also holds for the product of arbitrary finite such polynomials.
\end{rem}

A concept closely related to unimodality is log-concavity.
Let $\{a_i\}_{i=0}^n$ be a sequence of nonnegative numbers with no internal zeros, i.e.,
there are no three indices $i<j<k$ such that $a_i,a_k\neq 0$ and $a_j=0$.
It is {\it log-concave} if $a_{i-1}a_{i+1}\le a_i^2$ for $0<i<n$.
Clearly,
the sequence is log-concave if and only if %$\left\{a_{i+1}/a_i\right\}_{i=0}^{n-1}$ is nonincreasing, or equivalently,
$a_{i-1}a_{j+1}\le a_ia_j$ for $0<i\le j<n$.
Thus a log-concave sequence is unimodal.
A basic approach for attacking unimodality and log-concavity is to utilize the Newton's inequality:
If the polynomial $f(q)=\sum_{i=0}^{n}a_iq^i$ has only real zeros, then
$$a_i^2\ge a_{i-1}a_{i+1}\left(1+\frac{1}{i}\right)\left(1+\frac{1}{n-i}\right),\qquad i=1,2,\ldots,n-1$$
(see Hardy, Littlewood and P\'olya~\cite[p. 104]{HLP52}). It follows
that if all $a_i$ are nonnegative, then $f(q)$ is log-concave and
therefore unimodal. So, if $f(q)$ is a palindromic polynomial with
nonnegative coefficients and with only real zeros, then it is  a
$\La$-polynomial. Br\"{a}nd\'{e}n further showed that $f(q)$ is
actually $\B$-positive~\cite[Lemma 4.1]{Bra04-05}. We end this
section by showing the following stronger result.

\begin{thm}\label{thm-pf}
Let $f(q)=\sum_{j=0}^{n} a_jq^j$ be a palindromic polynomial of darga $n$ and
$$f(q)=\sum_{i=0}^{\lrf{n/2}}b_iq^i(1+q)^{n-2i}.$$
If the polynomial $f(q)=\sum_{j=0}^{n} a_jq^j$ has nonnegative coefficients and has only real zeros,
then so does the polynomial $\sum_{i=0}^{\lrf{n/2}}b_iq^i$.
In particular, $f(q)$ is $\B$-positive.
\end{thm}
\begin{proof}
%We may assume, without loss of generality,
%Clearly, it suffices to consider the case that $q\nmid f(q)$ and $(1+q)\nmid f(q)$.
We first consider the case that $q\nmid f(q)$ and $(1+q)\nmid f(q)$.
Now $f(q)$ is palindromic and has only real zeros,
so by Proposition~\ref{factor}, $n$ is even and
$$f(q)=a\prod_{k=0}^{n/2}(q+r_k)(q+1/r_k),$$
where $a>0, r_k>0$ and $r_k\neq 1$. It follows that
$$f(q)=a\prod_{k=0}^{n/2}\left[(1+q)^2+s_kq\right],$$
where $s_k=r_k+1/r_k-2>0$.
By Remark~\ref{prod-3}, we have
$$\sum_{i=0}^{n/2}b_iq^i=a\prod_{k=0}^{n/2}(1+s_kq),$$
which has nonnegative coefficients and has only real zeros.

For the general case, let $f(q)=q^r(1+q)^eg(q)$, where $q\nmid g(q)$ and $(1+q)\nmid g(q)$.
Then $g(q)$ is also a palindromic polynomial with nonnegative coefficients and with only real zeros.
By the discussion above, we have $g(q)=\sum_{i=0}^{\lrf{m/2}}c_iq^i(1+q)^{m-2i}$,
where $m=n-2r-e$ and $\sum_{i=0}^{\lrf{m/2}}c_iq^i$ is a polynomial with nonnegative coefficients and with only real zeros.
It follows that $$f(q)=\sum_{i=0}^{\lrf{m/2}}c_iq^{r+i}(1+q)^{m+e-2i}=\sum_{j=r}^{\lrf{(m+2r)/2}}c_{j-r}q^j(1+q)^{n-2j},$$
which implies that $b_j=c_{j-r}$ if $r\le j\le \lrf{(m+2r)/2}$ and $b_j=0$ otherwise.
So
$$\sum_{j=0}^{\lrf{n/2}}b_jq^j=q^r\sum_{i=0}^{\lrf{m/2}}c_iq^i,$$
which is a polynomial with nonnegative coefficients and with only real zeros, as required.
This completes our proof.
\end{proof}
\begin{rem}
It is well known that Eulerian polynomials, Narayana polynomials and derangement polynomials are palindromic and
have only real zeros (see, e.g., \cite{LW07}).
So these polynomials are all $\B$-positive.
See \cite{Sun13,SW14} for combinatorial interpretations to express these polynomials in terms of the basis $\B$ with nonnegative coefficients.
\end{rem}

\section{Remarks}
\hspace*{\parindent}

An important resource for $\La$-polynomials is rank-generating functions of posets.
We follow Engel~\cite{Eng97} %and Stanley~\cite[Chapter 3]{Sta97}
for poset notation and terminology.

Let $P$ be a finite graded poset of rank $n$
and $F(P;q)$ be the rank-generating function of $P$.
Let $\cc_n$ and $\cb_n$ be the chain of length $n$ and the Boolean algebra on $n$ elements respectively.
Then $F(\cc_n;q)=1+q+q^2+\cdots+q^n$ and $F(\cb_n;q)=(1+q)^n$.
It is well known that if $P$ has a symmetric chain decomposition, then
$$F(P;q)=\sum_{j=0}^{\lrf{n/2}} c_jq^j(1+q+\cdots+q^{n-2j}),$$
where $c_j$ is the number of symmetric chains of length $n-2j$ in the symmetric chain decomposition of $P$.
Similarly, if $P$ has a symmetric Boolean decomposition (see~\cite{Pet11} for details), then
$$F(P;q)=\sum_{j=0}^{\lrf{n/2}}b_jq^j(1+q)^{n-2j},$$
where $b_j$ is the number of symmetric Boolean algebras $\cb_{n-2j}$ in the decomposition of $P$.

More generally, let $P$ be a Peck poset.
Then $F(P;q)$ is a $\La$-polynomial.
Two particular interesting examples of Peck posets are Young lattices $L(m,n)$ and $M(n)$,
whose rank-generating functions are the Gaussian polynomial
$$F(L(m,n);q)=\binom{m+n}{n}_q$$
and the partition polynomial
$$F(M(n);q)=(1+q)(1+q^2)\cdots (1+q^n)$$
respectively.
%A long-standing conjecture is that both $L(m,n)$ and $M(n)$ are palindromic chain orders.
Stanley~\cite{Sta80b} showed, among other things,
that $L(m,n)$ and $M(n)$ are Peck posets using the hard Lefschetz theorem of algebraic geometry
(see \cite{Pro82} for a linear algebra approach).
This implies that the Gaussian polynomial and the partition polynomial are $\La$-polynomials.
O'Hara~\cite{Koh90} gave a constructive proof for the unimodality of the Gaussian polynomial.
However, there is no combinatorial interpretation for the unimodality of the partition polynomial $(1+q)(1+q^2)\cdots (1+q^n)$.

On the other hand, Odlyzko and Richmond~\cite{OR82} developed analytic techniques
to study the unimodality of polynomials $(1+q^{a_1})\cdots(1+q^{a_n})$,
and in particular, showed the unimodality of $(1+q)(1+q^2)\cdots (1+q^n)$.
Almkvist~\cite{Alm89} further proposed the following.

\begin{conj}
The polynomial
$$f_{n,r}(q)=\prod_{i=1}^n\frac{1-q^{ri}}{1-q^i}$$
is a $\La$-polynomial if either $r$ is even and $n\ge 1$ or $r$ is odd and $n\ge 11$.
\end{conj}

Almkvist also verified the conjecture for $n=3,4,\ldots,20,100,101$ using Odlyzko and Richmond's techniques.
But the general case is still open.

It is interesting to find out a combinatorial interpretation of the unimodality of $(1+q)(1+q^2)\cdots (1+q^n)$
and how to express $(1+q)(1+q^2)\cdots (1+q^n)$ in terms of the basis $\A$ with nonnegative coefficients.
%We refer the reader to \cite{Pet11,Sun13,SW14} for some related results.

\end{document}